\documentclass[11pt]{amsart}

\usepackage{amsfonts,amssymb,amsmath,amsthm,amscd,enumerate}
\usepackage{latexsym}

\newtheorem{theorem}{Theorem}[section]

\newtheorem{lemma}[theorem]{Lemma}
\newtheorem{corollary}[theorem]{Corollary}

\newtheorem*{remark}{Remark}

\newtheorem{conjecture}[theorem]{Conjecture}

\title{ON SOME PRODUCTS OF COMMUTATORS IN AN ASSOCIATIVE RING}

\author{GALINA DERYABINA}

\address{Department of Computational Mathematics and Mathematical Physics (FS-11), Bauman Moscow State Technical University, 2-nd Baumanskaya Street, 5, 105005 Moscow, Russia
\newline
galina\_deryabina@mail.ru}

\author{ALEXEI KRASILNIKOV}

\address{Departamento de Matem\'atica, Universidade de Bras\'\i lia, 70910-900 Bras\'\i lia, DF, Brasil
\newline
alexei@unb.br}

\date{}

\begin{document}

\maketitle

\begin{abstract}
Let $A$ be a unital associative ring and let $T^{(k)}$ be the two-sided ideal of $A$ generated by all commutators $[a_1, a_2, \dots , a_k]$ $(a_i \in A)$ where $[a_1, a_2] = a_1 a_2 - a_2 a_1$, $[a_1, \dots , a_{k-1}, a_k] = \bigl[ [a_1, \dots , a_{k-1}], a_k \bigr]$ $(k >2)$. It has been known that, if either $m$ or $n$ is odd then
\[
6 \, [a_1, a_2, \dots , a_m] [b_1, b_2, \dots , b_n] \in T^{(m+n-1)}
\]
for all $a_i, b_j \in A$. This was proved by Sharma and Srivastava in 1990 and independently rediscovered later (with different proofs) by various authors. The aim of our note is to give a simple proof of the following result: if at least one of the integers $m,n$ is odd then, for all $a_i, b_j \in A$,
\[
3 \, [a_1, a_2, \dots , a_m] [b_1, b_2, \dots , b_n] \in T^{(m+n-1)}.
\]
Since it has been known that, in general,
\[
[a_1, a_2, a_3] [b_1, b_2] \notin T^{(4)},
\]
our result cannot be improved further for all $m, n$ such that at least one of them is odd.
\end{abstract}



%%%%%%%%%%%%%%%%%%%%%
%%%%%%%%%%%%%%%%%%%%%
\section{Introduction}

Let $A$ be a unital associative ring. Define a left-normed commutator $[a_1, \dots , a_{k-1}, a_k]$ ($k >1$, $a_i \in A$ for all $i$) recursively as follows: $[a_1, a_2] = a_1 a_2 - a_2 a_1$, $[a_1, \dots , a_{k-1}, a_k] = \bigl[ [a_1, \dots , a_{k-1}], a_k \bigr]$ $(k >2)$. Let $T^{(k)} = T^{(k)} (A)$ be the two-sided ideal of $A$ generated by all commutators $[a_1, a_2, \dots , a_k]$ $(a_i \in A)$.

For each unital associative ring $A$, all $m, n \in \mathbb Z$, $m, n > 1$ and all $a_i, b_j \in A$, we have
\begin{equation}
\label{[m][n](mn2)}
[a_1, \dots , a_m] [b_1, \dots , b_n] \in T^{(m + n -2)}.
\end{equation}
This result was proved by Latyshev \cite[Lemma 1]{Latyshev65} in 1965 in an article published in Russian and independently rediscovered by Gupta and Levin \cite[Theorem 3.2]{GL83} in 1983.

If $m = 3$ then, by (\ref{[m][n](mn2)}), we have $[a_1,a_2,a_3][b_1, b_2, \dots , b_n] \in T^{(n+1)}$. Note that $T^{(n+1)} \supseteq T^{(n+2)}$. In 1985 Levin and Sehgal \cite[Lemma 2(a)]{LS85} proved that, for all $n > 1$ and all  $a_i, b_j \in A$,
\begin{equation*}
\label{[3][n]}
3 \, [a_1,a_2,a_3][b_1, b_2, \dots , b_n] \in T^{(n+2)} .
\end{equation*}
Earlier, in 1978, a similar result was proved by Volichenko  \cite[Lemma 1]{Volichenko78} in a preprint written in Russian. More recently, some particular cases of this result were independently rediscovered, with different proofs, in \cite[Theorem 3.4]{EKM09} and \cite[Lemma 1]{Gordienko07}.

In 1990 Sharma and Srivastava \cite[Theorem 2.8]{SS90} proved that if  $m, n \in \mathbb Z$, $m, n >0$ and at least one of  $m,n$ is odd then, for all $a_i, b_j \in A$,
\begin{equation}
\label{6[m][n]}
6 \, [a_1, \dots , a_m] [b_1, \dots , b_n] \in T^{(m+n-1)} .
\end{equation}
Recently this result was independently rediscovered, with different proofs, in \cite[Corollary 1.4]{BJ10} and \cite[Theorem 1]{GP15}.

Note that if $m = 2m'$, $n = 2n'$ both are even then the result similar to (\ref{6[m][n]}) does not hold: there exit associative algebras $A = A(m,n)$ over a field $F$ of characteristic $0$ such that, for some $a_i, b_j \in A$,
\[
[a_1, \dots , a_{2m'} ] [b_1, \dots , b_{2n'}] \notin T^{(2m' + 2n' -1)}
\]
(see  \cite[Theorem 1.4]{DK17} or \cite[Lemma 6]{GP15}) so
\begin{equation}
\label{2m2n}
\ell \, [a_1, \dots , a_{2m'} ] [b_1, \dots , b_{2n'}] \notin T^{(2m' + 2n' -1)}
\end{equation}
for all $\ell \in \mathbb Z$, $\ell \ne 0$.

The aim of the present note is to give a simple proof of the following theorem that improves the result (\ref{6[m][n]}) by Sharma and Srivastava.

\begin{theorem}
\label{theorem3[aa][bb]}
Let $A$ be a unital associative ring. Let $m, n \in \mathbb Z$, $m, n >0$. Suppose that at least one of the integers $m,n$ is odd. Then, for all $a_i, b_j \in A$,
\[
3 \, [a_1, \dots , a_m] [b_1, \dots , b_n] \in T^{(m+n-1)} .
\]
\end{theorem}

Note that, for some $A$ and some $a_i, b_j \in A$,
\[
[a_1, a_2, a_3] [b_1, b_2] \notin T^{(4)}
\]
(see \cite[Theorem 1.1]{Kr13}) so for $m = 3$, $n=2$ Theorem \ref{theorem3[aa][bb]} cannot be improved. On the other hand, if $m = n = 3$ then, for each associative ring $A$ and all $a_i, b_j \in A$,
\[
[a_1, a_2, a_3] [b_1, b_2, b_3] \in T^{(5)}
\]
(see \cite[Lemma 2.1]{CostaKras18}). There is some evidence that suggests that the case $m = n = 3$ is exceptional and in all other cases Theorem \ref{theorem3[aa][bb]} cannot be improved further.

\begin{conjecture}
\label{conjecture[aa][bb]}
Let $m, n \in \mathbb Z$, $m, n >1$ and either $m$ or $n$ is odd. If $(m, n) \ne (3, 3)$ then there is a unital  associative ring $A$ and $a_i, b_j \in A$ such that
\[
[a_1, \dots ,a_m][b_1,  \dots , b_n] \notin T^{(m+n-1)}
\]
\end{conjecture}

It is easy to check that to prove (or disprove) Conjecture \ref{conjecture[aa][bb]} one can assume that $A = \mathbb Z \langle X, Y \rangle$ is the free unital associative ring on a free generating set $X \cup Y$ where $X = \{ x_1, x_2, \dots \}$, $Y = \{ y_1, y_2, \dots \}$ and $a_i = x_i$, $b_j = y_j$ for all $i, j$.

\medskip
The following assertion follows immediately from Theorem \ref{theorem3[aa][bb]} (see \cite[Prop. 1.3 and 1.4]{DK17-2} for more details).

\begin{corollary}
Let $A$ be a unital associative ring and let $n_1, \dots , n_k \in \mathbb Z$, $n_i > 0$ for all $i$. Suppose that $\ell$ of the integers $n_1, \dots , n_k$ are odd. Let
\[
q = \left\{
\begin{array}{ll}
\ell & \mbox{ if $\ell < k$, that is, if at least one $n_i$ is even;}
\\
\ell - 1 & \mbox{ if $\ell = k$, that is, if all $n_i$ are odd }
\end{array}
\right .
\]
and let $N_{k,q} = n_1 + \dots + n_k - 2(k-1) + q$.  Then, for all $a_{ij} \in A$,
\[
3^q \, [a_{11}, \dots , a_{1n_1}] \dots [a_{k1}, \dots , a_{kn_k}] \in T^{(N_{k,q})} .
\]
\end{corollary}

Note that, in general,
\[
r \, [a_{11}, \dots , a_{1n_1}] \dots [a_{k1}, \dots , a_{kn_k}] \notin T^{(N_{k,q} + 1)}
\]
for some unital associative ring $A$, some $a_{ij} \in A$ and all $r \in \mathbb Z$, $r \ne 0$. This was proved by Dangovski \cite[Prop. 2.2]{Dangovski15} if $\ell = k$ and by the authors of the present article \cite[Theorem 1.7]{DK17-2} if $\ell < k$.

%%%%%%%%%%%%%%%%%%%%%
\medskip
\noindent
\textbf{Remarks.} 1. Theorem \ref{theorem3[aa][bb]} and Conjecture \ref{conjecture[aa][bb]} are closely connected to the description of the additive group of the ring $\mathbb Z \langle X \rangle / T^{(k)}$ where $\mathbb Z \langle X \rangle$ is the free unital associative ring with a free generating set $X = \{ x_1, x_2, \dots \}$.

It is clear that the additive group of the ring $\mathbb Z\langle X \rangle / T^{(2)}$ is free abelian. It was shown in \cite{BEJKL12} that the additive group of $\mathbb Z \langle X \rangle / T^{(3)}$ is also free abelian. On the other hand, the additive group of the ring $\mathbb Z \langle X \rangle / T^{(4)}$ is a direct sum $G \oplus H$ of a free abelian group $G$ and an elementary abelian $3$-group $H$ (see \cite{DK15,Kr13}). Computational data by Cordwell, Fei and Zhou presented in \cite[Appendix A]{CFZ15} suggest that for $k \ge 6$ the additive group of the ring $\mathbb Z \langle X \rangle / T^{(k)}$  is also a direct sum of a free abelian group $G$ and a non-trivial elementary abelian $3$-group $H$ while for $k = 5$ this group is free abelian. However, it is still an open problem whether the torsion subgroup $H$ of the additive group of  $\mathbb Z \langle X \rangle / T^{(k)}$ is indeed a non-trivial (elementary) abelian $3$-group if $k >5$  and $H = \{ 0 \}$ if $k = 5$.

If Conjecture \ref{conjecture[aa][bb]} holds then the elements $[a_1, \dots , a_m][b_1, \dots , b_n] + T^{(m+n-1)}$ are non-trivial elements of $H \subset \mathbb Z \langle X \rangle / T^{(m+n-1)}$ whose order, by Theorem \ref{theorem3[aa][bb]}, is equal to $3$ (if $( m, n) \ne (3, 3)$). If $m = 3$, $n = 2$ then such products of commutators generate $H$ as a two-sided ideal in $\mathbb Z \langle X \rangle / T^{(4)}$ (see \cite{DK15,Kr13}). One might expect a similar situation if $k>5$.

2. The proof of (\ref{6[m][n]}) given in  \cite{BJ10} can be modified to prove Theorem \ref{theorem3[aa][bb]} (see \cite[Remark 3.9]{AE15} for explanation). This modification uses computer calculations in a free associative ring. Our proof of Theorem \ref{theorem3[aa][bb]} does not require computer calculations and is much simpler then that modification of the proof given in  \cite{BJ10}.

3. Theorem \ref{theorem3[aa][bb]} and its corollary remain valid for a non-unital associative ring $A$; one can easily deduce this from the corresponding results for unital rings. We state and prove our results for a unital associative ring $A$ in order to simplify notation in the proof.  

%%%%%%%%%%%%%%%%%%%%%
%%%%%%%%%%%%%%%%%%%%%
\section{Proof of Theorem \ref{theorem3[aa][bb]} }

It is straightforward to check that
\begin{equation*}
\label{[abc]}
[a_1 a_2, b] = a_1[a_2, b] + [a_1, b] a_2,
\end{equation*}
so the map $D_b: A \rightarrow A$ such that $D_b (a) = [a, b]$ $(a \in A)$ is a derivation of the ring $A$. It follows that
\begin{equation}
\label{[aabb]}
[a_1 a_2, b_1, b_2]  = a_1 [ a_2, b_1, b_2] + [a_1, b_1] [a_2, b_2] + [a_1, b_2] [a_2, b_1] + [a_1, b_1, b_2] a_2
\end{equation}
\begin{multline}
\label{[aabbb]}
[a_1 a_2, b_1, b_2, b_3] = a_1 [a_2, b_1, b_2, b_3] + [a_1, b_1] [a_2, b_2, b_3] +[a_1, b_2] [a_2, b_1, b_3]
\\
+ [a_1, b_3] [a_2, b_1, b_2] + [a_1, b_1, b_2] [a_2, b_3] + [a_1, b_1, b_3] [a_2, b_2]
\\
+ [a_1, b_2, b_3] [a_2, b_1] + [a_1, b_1, b_2, b_3] a_2
\end{multline}
for all $a_i,b_j \in A$.

%%%%%%%%%%%%%%%%%%%%%
The following lemma is a modification of well-known results (see, for instance, \cite[Lemma 2.1]{Kr13}, \cite[Lemma 2 (3)]{Latyshev65}, \cite[Lemma 8.2]{MAP17}). We prove it here in order to have the article self-contained.

\begin{lemma}
\label{k3Tk+2}
Let $A$ be a unital associative ring. Then, for all $k>1$ and all $g_i, f_j \in A$, we have

\begin{equation}
\label{32+32T4-1}
[g_{1}, \dots , g_{k - 1},f_{1}][f_{2}, f_{3}, f_4] +
[g_{1}, \dots , g_{k - 1},f_{2}][f_{1}, f_{3}, f_4] \in T^{(k + 2)},
\end{equation}
\begin{equation}
\label{32+32T4-2}
[g_{1}, \dots , g_{k - 1},f_{1}][f_{2}, f_{3}, f_4] +
[g_{1}, \dots , g_{k - 1},f_{4}][f_{2}, f_{3}, f_1] \in T^{(k + 2)}.
\end{equation}

\end{lemma}

%%%%%%%%%%%%%%%%%%%%%
\begin{proof}
Let $c_1 = [g_1, \dots , g_{k - 1}, g_k]$. It is clear that $ [c_1, (h_1h_2), h_3] \in T^{(k + 2)}$ for all $g_i, h_j \in A$. We have $[c_1, (h_1 h_2), h_3] = - [h_1 h_2, c_1, h_3]$ and, by (\ref{[aabb]}),
\[
[h_1 h_2, c_1, h_3] = h_1[h_2, c_1, h_3] + [h_1, c_1][h_2, h_3] + [h_1, h_3][h_2, c_1] + [h_1, c_1, h_3] h_2 \in T^{(k + 2)}.
\]
It is clear that $h_1 [h_2, c_1, h_3] = - h_1 [ c_1, h_2, h_3]  \in T^{(k + 2)}$; similarly, $ [h_1, c_1, h_3] h_2 \in T^{(k + 2)}$  so
\[
[h_1, c_1][h_2, h_3] + [h_1, h_3][h_2, c_1] \in T^{(k + 2)}.
\]
It follows that
\[
[c_1, h_{1}][h_{2},h_{3}] + [c_{1},h_{2}][h_{1},h_{3}] - \bigl[ [c_{1},h_{2}], [h_{1},h_{3}] \bigr]
= - \bigl( [h_1, c_1][h_2, h_3] + [h_1, h_3][h_2, c_1] \bigr) \in T^{(k + 2)} .
\]
Since
\[
\bigl[ [c_1, h_2], [h_1, h_3] \bigr] = [c_1, h_2, h_1, h_3] - [c_1, h_2, h_3, h_1]  \in T^{(k +3)} \subseteq T^{(k + 2)},
\]
we have
\begin{equation}
\label{[ch][hh]}
[c_1, h_{1}][h_{2},h_{3}] + [c_{1},h_{2}][h_{1},h_{3}] \in T^{(k + 2)}.
\end{equation}

Now we check that (\ref{32+32T4-2}) holds. Let $c_2 = [g_1, \dots , g_{k - 1}]$. It is clear that $[c_2, (f_1 f_2), f_3, f_4] \in T^{(k + 2)}$ for all $g_i, f_j \in A$. We have $[c_2, (f_1 f_2), f_3, f_4] = - [f_1 f_2, c_2, f_3, f_4]$ and, by (\ref{[aabbb]}),
\begin{multline*}
[f_1 f_2, c_2, f_3, f_4] = f_1 [f_2, c_2, f_3, f_4] + [f_1, c_2] [f_2, f_3, f_4] + [f_1, f_3] [f_2, c_2, f_4]
\\
+ [f_1, f_4] [f_2, c_2, f_3] + [f_1, c_2, f_3][f_2, f_4] + [f_1, c_2, f_4] [f_2, f_3]
\\
+ [f_1, f_3, f_4] [f_2, c_2] + [f_1, c_2, f_3, f_4] f_4 \in T^{(k + 2)}.
\end{multline*}
It is clear that $f_1 [f_2, c_2, f_3, f_4] = - f_1 [c_2, f_2, f_3, f_4] \in T^{(k + 2)}$ and similarly $[f_1, c_2, f_3, f_4] f_4 \in T^{(k +2)}$. Further, by (\ref{[ch][hh]}),
\begin{multline*}
[f_1, f_3] [f_2, c_2, f_4] + [f_1, f_4] [f_2, c_2, f_3] = \bigl( [c_2, f_2, f_4][f_3, f_1] + [c_2, f_2, f_3] [f_4, f_1] \bigr)
\\
- \bigl[ [c_2, f_2, f_4], [f_3, f_1] \bigr] - \bigl[ [c_2, f_2, f_3], [f_4, f_1] \bigr] \in T^{(k +2)}
\end{multline*}
and similarly
$[f_1, c_2, f_3][f_2, f_4] + [f_1, c_2, f_4] [f_2, f_3] = [c_2, f_1, f_3][f_4, f_2] + [c_2, f_1, f_4] [f_3, f_2] \in T^{(k + 2)}$. It follows that
\[
[f_1, c_2] [f_2, f_3, f_4] + [f_1, f_3, f_4] [f_2, c_2] \in T^{(k + 2)}
\]
so
\begin{multline*}
[c_2, f_1] [f_2, f_3, f_4] + [c_2, f_2] [f_1, f_3, f_4] - \bigl[ [c_2, f_2], [f_1, f_3, f_4] \bigr]
\\
= - \bigl( [f_1, c_2] [f_2, f_3, f_4] + [f_1, f_3, f_4] [f_2, c_2] \bigr)  \in T^{(k + 2)}.
\end{multline*}
Since $\bigl[ [c_2, f_2], [f_1, f_3, f_4] \bigr] \in T^{(k + 3)} \subseteq T^{(k + 2)}$, we have
\[
[c_2, f_1] [f_2, f_3, f_4] + [c_2, f_2] [f_1, f_3, f_4] \in T^{(k + 2)}
\]
so (\ref{32+32T4-2}) holds.

It remains to check that (\ref{32+32T4-1}) holds. Recall that $c_2 = [g_1, \dots , g_{k-1}]$. By the Jacobi identity, \[
[f_2, f_3, f_4] = [f_4, f_3, f_2] - [f_4, f_2, f_3], \qquad [f_2, f_3, f_1] = [f_1, f_3, f_2] - [f_1, f_2, f_3]
\]
so
\begin{multline*}
[c_2, f_1] [f_2, f_3, f_4] + [c_2, f_4] [f_2, f_3, f_1]
\\
= \bigl( [c_2, f_1] [f_4, f_3, f_2] + [c_2, f_4] [f_1, f_3, f_2] \bigr) - \bigl( [c_2, f_1] [f_4, f_2, f_3] + [c_2, f_4] [f_1, f_2, f_3] \bigr) .
\end{multline*}
By (\ref{32+32T4-2}), we have $[c_2, f_1] [f_2, f_3, f_4] + [c_2, f_4] [f_2, f_3, f_1]  \in T^{(k + 2)}$, that is, (\ref{32+32T4-1}) holds.

This completes the proof of Lemma \ref{k3Tk+2}.
\end{proof}

%%%%%%%%%%%%%%%%%%%%%
\begin{corollary}
\label{[k][3]sigma}
Let $A$ be a unital associative ring. Then, for all $k>1$ and all $g_i, f_j \in A$ and for each permutation $\sigma$ on the set $\{ 1, 2, 3, 4 \}$, we have
\begin{equation}
\label{sigma}
[g_1, \dots , g_{k-1}, f_{\sigma (1)}] [f_{\sigma (2)}, f_{\sigma (3)}, f_{\sigma (4)}]
\equiv (-1)^{\sigma} [g_1, \dots , g_{k-1}, f_{1}] [f_{2}, f_{3}, f_{4}] \pmod{T^{(k + 2)}} .
\end{equation}
\end{corollary}

%%%%%%%%%%%%%%%%%%%%%
\begin{proof}
It is clear that (\ref{sigma}) is true if $\sigma = (23)$ is the transposition that permutes $2$ and $3$. By Lemma \ref{k3Tk+2}, (\ref{sigma}) holds if $\sigma = (12)$ and $\sigma = (14)$. Hence, (\ref{sigma}) is true for all permutations $\sigma$ that are products of the transpositions $(12), (23)$ and $(14)$. However, it is easy to check that these $3$ transpositions generate the group $S_4$ of all permutations on the set $\{ 1, 2, 3, 4 \}$. The result follows.
\end{proof}

%%%%%%%%%%%%%%%%%%%%%
The following corollary has been proved by Levin and Sehgal \cite[Lemma 2(a)]{LS85}. Our proof is different from one given in \cite{LS85}; it shows that the coefficient $3$ in (\ref{[3][n]}) appears because of the Jacobi identity.
\begin{corollary}[see \cite{LS85}]
\label{3[k][3]}
Let $A$ be a unital associative ring. Then, for all $a_i, b_j \in A$,
\begin{equation}
\label{[3][n]}
3 \, [a_1, \dots , a_k]  [b_1, b_2, b_3] \in T^{(k+2)}.
\end{equation}
\end{corollary}

%%%%%%%%%%%%%%%%%%%%%
\begin{proof}
By the Jacobi identity,
\[
[a_1, \dots , a_k] \bigl( [b_1, b_2, b_3] + [b_2, b_3, b_1] + [b_3, b_1, b_2]) = 0.
\]
On the other hand, by Corollary \ref{[k][3]sigma},
\[
[a_1, \dots , a_k] [b, b_2, b_3]  \equiv [a_1, \dots , a_k] [b_2, b_3, b_1] \pmod{T^{(k+2)}}
\\
\equiv [a_1, \dots , a_k] [b_3, b_1, b_2]) \pmod{T^{(k+2)}} .
\]
It follows that
\[
3 \, [a_1, \dots , a_k]  [b_1, b_2, b_3] \equiv 0 \pmod{ T^{(k+2)} },
\]
as required.
\end{proof}

%%%%%%%%%%%%%%%%%%%%%
The following lemma is a modification of \cite[Lemma 2]{GP15}. Note that the proof given in \cite{GP15} allows to prove only the inclusion $6 \, [ T^{(k)}, A, A ] \subseteq T^{(k+2)}$.

\begin{lemma}[cf. \cite{GP15}]
\label{lemma3[TnAA]}
Let $A$ be a unital associative ring. Then, for each $k \ge 1$, we have
\[
3 \, [ T^{(k)}, A, A ] \subseteq T^{(k+2)}.
\]
\end{lemma}

%%%%%%%%%%%%%%%%%%%%%
\begin{remark}
{ \rm
In general, for $\ell \in \mathbb Z$, $\ell \ne 0$,
\[
\ell \, [T^{(k)} , A] \nsubseteq T^{(k + 1)}.
\]
More precisely, if $k = 2k'$ is even then, in general,
\[
\ell \, [T^{(2k')} , A] \nsubseteq T^{(2k' + 1)}
\]
for any $\ell \in \mathbb Z$, $\ell \ne 0$. This can be easily deduced from (\ref{2m2n}).

On the other hand, if $k = 2k' + 1$ is odd then
\[
3 \, [T^{(2k' + 1)} , A] \subseteq T^{(2k' + 2)} ;
\]
however, to prove this one has to use Theorem \ref{theorem3[aa][bb]}.
}
\end{remark}

%%%%%%%%%%%%%%%%%%%%%
\begin{proof}[Proof of Lemma \ref{lemma3[TnAA]}]
By definition, $T^{(k)}$ is the two-sided ideal of $A$ generated by all commutators $[a_1, a_2, \dots , a_k]$ $(a_i \in A)$. However, one can easily check that $T^{(k)}$ is generated by the commutators $[a_1, a_2, \dots , a_k]$ $(a_i \in A)$ as a \textit{right ideal} in $A$  as well. Hence, to prove the lemma it suffices to prove that
\[
3 \, [cu, v, w] \in T^{(k+2)}
\]
where $u,v,w \in A$, $c = [a_1, a_2, \dots , a_k]$ $(a_i \in A)$. We have
\begin{equation}
\label{[cuvw]}
[cu, v, w] = c[u,v,w] + [c,v][u,w] + [c,w][u,v] + [c,v,w] u .
\end{equation}
By Corollary \ref{3[k][3]}, we have
\begin{equation}
\label{3c[uvw]}
3 \, c [u,v, w] \in T^{(k+2)}.
\end{equation}
It is clear that
\begin{equation}
\label{[cvw]u}
[c, v, w] u \in T^{(k+2)}.
\end{equation}
Further,
\[
[c,v] [u,w] + [c,w][u,v] = [v,c] [w, u] + [v,u] [w,c] + \bigl[ [c,w], [u,v] \bigr]
\]
where $\bigl[ [c,w], [u,v] \bigr]  \in T^{(k+3)} \subseteq T^{(k+2)}$ and, by (\ref{[aabb]}),
\begin{multline*}
[v,c] [w, u] + [v,u] [w,c] = [vw, c, u] - v[w, c, u] -  [v, c, u] w
\\
= - [c, vw, u] + v[c,w,u] + [c,v,u] w \in T^{(k+2)}.
\end{multline*}
Hence,
\begin{equation}
\label{[cv][uw]}
[c,v] [u,w] + [c,w][u,v] \in T^{(k+2)}.
\end{equation}
It follows from (\ref{[cuvw]})--(\ref{[cv][uw]}) that $3 \, [cu, v, w] \in T^{(k+2)}$, as required. The  proof of Lemma~\ref{lemma3[TnAA]} is completed.
\end{proof}

%%%%%%%%%%%%%%%%%%%%%%
Now we are in a position to prove the theorem. We need to check that
\[
3 \, [a_1, \dots , a_m] [b_1, \dots , b_n] \in T^{(m+n-1)}
\]
if either $m$ or $n$ is odd. Since  $\bigl[ [a_1, \dots , a_m], [b_1, \dots , b_n] \bigr] \in T^{(m+n)} \subseteq T^{(m+n-1)}$, we may assume without loss of generality that $m$ is odd.

The proof is by induction on $m$. If $m=1$ then the theorem clearly holds. Let $m = 2k+1$, $k > 0$. Suppose that
\[
3 \, [a_1, \dots , a_{m-2}] [b_1, \dots , b_n] \in T^{(m+n-3)}
\]
for all $n>0$ and all $a_i, b_j \in A$.

Let $c_1 = [a_1, a_2, \dots , a_{m-2}]$, $c_2 = [b_1, b_2, \dots , b_n]$ $(a_i, b_j \in A)$. By (\ref{[aabb]}),   we have
\[
[c_1 c_2, v, w] = [c_1, v, w] c_2 + [c_1, v] [c_2, w] + [c_1, w] [c_2, v] + c_1 [c_2, v, w]
\]
so
\begin{equation}
\label{[ccvw]}
[c_1, v, w] c_2 = [c_1 c_2, v, w]  - \bigl( [c_1, v] [c_2, w] + [c_1, w] [c_2, v] \bigr) -   c_1 [c_2, v, w]
\end{equation}
for all $v,w \in A$.

By the induction hypothesis, $3 \, c_1 c_2 \in T^{(m + n -3)}$ so, by Lemma \ref{lemma3[TnAA]},
\begin{equation}
\label{3[ccvw]}
3 \, [c_1 c_2, v, w] \in 3 \, [T^{(m+n-3)}, A, A] \subset T^{(m+n-1)} .
\end{equation}
Again, by the induction hypothesis,
\begin{equation}
\label{3c[cvw]}
3 \, c_1 \, [c_2, v, w] \in T^{(m + n - 1)} .
\end{equation}
Further,
\[
 [c_1, v] [c_2, w] + [c_1, w][c_2, v] = [v, c_1] [w, c_2] + [v, c_2] [w, c_1] + \bigl[ [c_1, w], [c_2, v] \bigr]
\]
where $\bigl[ [c_1, w], [c_2, v] \bigr] \in T^{(m + n)} \subseteq T^{(m+ n -1)}$ and, by (\ref{[aabb]}),
\[
[v, c_1] [w, c_2] + [v, c_2] [w, c_1]  =  [v w, c_1, c_2] - v [w, c_1, c_2] - [v, c_1, c_2] w  \in T^{(m + n -1)} .
\]
Hence,
\begin{equation}
\label{[cv][cw]}
 [c_1, v] [c_2, w] + [c_1, w][c_2, v]  \in T^{(m + n -1)} .
\end{equation}
It follows from (\ref{[ccvw]})--(\ref{[cv][cw]}) that $3 \, [c_1, v, w] c_2 \in T^{(m + n - 1)} $ for all $v, w \in A$. Theorem \ref{theorem3[aa][bb]} follows.

%%%%%%%%%%%%%%%%%%%%
\section*{Acknowledgment}

The second author was partially supported by CNPq grant 310331/2015-3.

%%%%%%%%%%%%%%%%%%%%
%%%%%%%%%%%%%%%%%%%%

\end{document}